\documentclass[11pt]{amsart}%{scrartcl}%{amsart}

\usepackage{euscript}
\usepackage{amsmath}
\usepackage{amsthm}
\usepackage{epsfig}
\usepackage{amssymb}\usepackage{amscd}
\usepackage{epic}
\usepackage{eufrak}
\usepackage{marginnote}
\usepackage{color}

\numberwithin{equation}{section}
\numberwithin{equation}{subsection}

\theoremstyle{plain}
\newtheorem{theorem}[equation]{Theorem}
\newtheorem{lemma}[equation]{Lemma}

\newtheorem{prop}[equation]{Proposition}
\newtheorem{clm}[equation]{Claim}

\theoremstyle{definition}
\newtheorem{example}[equation]{Example}
\newtheorem{remark}[equation]{Remark}

\newtheorem{defn}[equation]{Definition}

\newtheorem{rem}[equation]{Remark}

\numberwithin{equation}{section}
\numberwithin{equation}{subsection}

\def\C{\mathbb C}
\def\Q{\mathbb Q}
\def\R{\mathbb R}
\def\Z{\mathbb Z}

\newcommand{\frG}{{\frak G}}

\newcommand{\calf}{{\mathcal F}}
\newcommand{\calO}{{\mathcal O}}

\newcommand{\calL}{\mathcal{L}}

\newcommand{\labelpar}{\label}

\newcommand{\frsw}{\mathfrak{sw}}

\title{Seiberg--Witten invariant
of the universal abelian cover of $S^3_{-p/q}(K)$}

\author{J\'{o}zsef Bodn\'{a}r}
\address{A. R\'enyi Institute of Mathematics, 1053 Budapest,
Re\'altanoda u. 13-15,  Hungary.}
\email{bodnar.jozef@renyi.mta.hu}

\author{Andr\'as N\'emethi}
\address{A. R\'enyi Institute of Mathematics, 1053 Budapest, Re\'altanoda u. 13-15, Hungary.}
\email{nemethi.andras@renyi.mta.hu}
\thanks{The first author is supported by
 the `Lend\"ulet' and ERC program `LTDBud' at MTA Alfr\'ed R\'enyi Institute of Mathematics.
The second  author is partially supported by OTKA Grants 100796 and K112735. }

\date{}

\begin{document}

\maketitle

%\begin{center}
% {\em Dedicated to }
%\end{center}

\pagestyle{myheadings} \markboth{{\normalsize
J. Bodn\'ar and A. N\'emethi}}{ {\normalsize Seiberg--Witten invariant
of the universal abelian cover of $S^3_{-p/q}(K)$}}

\begin{abstract}
We prove an additivity property for the normalized Seiberg--Witten invariants with respect to
the universal abelian cover of those 3--manifolds, which are obtained via
 negative rational Dehn surgeries along connected sum of algebraic knots. Although the statement is purely topological, we use the theory of complex singularities in several steps of the proof.
This topological covering additivity property can be compared with certain analytic
properties of  normal surface singularities,  especially with functorial behaviour
 of the (equivariant) geometric genus of singularities. We present several examples
  in order to find  the validity limits of the proved property, one of them shows
  that the covering additivity property is not true for negative definite plumbed
  3--manifolds in general.
\end{abstract}

\section{Introduction}\labelpar{s:intro}

\subsection{Motivation}\labelpar{ss:motiv}

In this paper we prove an additivity property of the 3--dimensional (normalized)
Seiberg--Witten invariant
with respect to an abelian cover, valid  for surgery 3--manifolds. Namely,
assume that $M$ is obtained
as a negative rational surgery along connected sum of algebraic knots in the three-sphere $S^3$. Let
$\Sigma$ be its universal abelian cover. Theorem \ref{thm:main} states that the sum
over all spin$^c$ structures of the
Seiberg--Witten invariants of $M$ (after
normalisation) equals to the canonical Seiberg--Witten invariant of $\Sigma$.

Both covers of manifolds, and manifolds of form $S^3_{-p/q}(K)$, are extensively studied in recent articles. The stability of certain properties and invariants with respect to the
coverings is a key classical strategy in topology, it is even more motivated
by the recent proof of Thurston's virtually fibered conjecture \cite{Agol,Wise}.
Manifolds of form $S^3_{-p/q}(K)$ can be particularly interesting due to theorem of Lickorish and Wallace \cite{Lic,Wal} stating that every closed oriented three-manifold can be expressed as surgery on a \emph{link} in $S^3$. Based on this result, one can ask which manifolds have surgery representations with some restrictions. For example,
 using Heegaard--Floer homology, \cite{HKL} provides
 necessary conditions on manifolds having surgery representation along a \emph{knot}. In this context, Theorem \ref{thm:main} can be viewed also as a criterion for a manifold having surgery representation of form $S^3_{-p/q}(K)$ with $K$
 a connected sum of algebraic knots.

In fact, Seiberg--Witten invariants (SW) and Heegaard--Floer homologies are closely related. The SW invariants were originally introduced by Witten in \cite{Witten}, but they also arise as Euler characteristics of Heegaard--Floer homologies, cf. \cite{OSz1,OSz2}. In this article we will involve another cohomology theory with similar property. Since $S^3_{-p/q}(K)$ is
 representable by a negative definite plumbing graph, via \cite{NemSW} we view the SW invariants
as Euler characteristics of lattice cohomologies introduced in \cite{Nlatnorm}. The big
advantage of the lattice cohomology over the classical definition of Heegaard--Floer homology
is that it is computable algorithmically from the plumbing graph. In the last section
of applications and examples the above `covering additivity property' will be combined  with
results involving lattice cohomology.

Another strong motivation to study the above property is provided by the
theory of  complex normal surface singularities: the geometric genus of the analytic germ
is conjecturally connected with the SW invariant of the link of the germ.
Since the geometric genus  satisfies the `covering
additivity property' (cf. \S \ref{ss:sing}),
it is natural to ask for the validity of similar property at purely topological level.
Furthermore, from the point of view of singularity theory, the motivation for the surgery manifolds
 $S^3_{-p/q}(K)$ is also strong:
the link of the so called \emph{superisolated singularities}
(introduced in \cite{Lue}) are of this form.
 These  singularities are the key test-examples and  provide  counterexamples
for several conjectures. They embed the theory of projective plane curves to the theory
of surface singularities. For their brief introduction see Example \ref{ex:SIS},
for a detailed presentation see \cite{Lue,LMN}.

All these connections with the analytic theory will be used deeply in several points of the proof. For consequences of the main result regarding analytic invariants see the last sections.

%All these analytic connections will be used deeply in several points in the proof. For analytic consequences
%of the main result see the last sections.

\subsection{Notations.}\labelpar{ss:not} \
We recall some facts about negative definite plumbed 3-- and 4--manifolds, their spin$^c$ structures and Seiberg--Witten invariants.
 For more see  \cite{NemnicI,NemnicIII}.

Let $M$ be a $3$--manifold which is a rational homology sphere ($\mathbb{Q}HS^3$). Assume
that it has a negative definite plumbing representation with a decorated connected  graph $G$ with vertex set $\mathcal{V}$.
 In particular, $M$ is the boundary of a plumbed $4$--manifold $P$,
which is obtained by plumbing disc bundles over oriented surfaces $E_v\simeq S^2$, $v \in \mathcal{V}$
(according to  $G$), and which
has a negative definite intersection form.
A vertex  $v \in \mathcal{V}= \mathcal{V}(G)$
 is  decorated by  the \emph{self--intersection} $e_v \in \mathbb{Z}$ (of $E_v$ in $P$). In other words, $e_v$ is the euler number of the disc bundle over $E_v \cong S^2$.
 Since $M$ is a $\mathbb{Q}HS^3$, the graph   $G$ is a tree.
 We set $\#\mathcal{V}(G)$ for the number of vertices of $G$.

Below all the (co)homologies are considered with ${\mathbb Z}$--coefficients.

Denote by $L = L_G= \mathbb{Z}\langle E_v \rangle_{v\in \mathcal{V}}$ the free abelian group
generated by basis elements $E_v$, indexed by ${\mathcal V}$.
It can be identified with $H_2(P)$,  where $E_v$ represent the zero sections of the disc--bundles.
It carries the negative definite intersection form $(.,.) = (.,.)_G$ (of $P$; readable from $G$ too).
This form naturally extends to $L \otimes \mathbb{Q}$.
Denoting by $L' = L_G' = \textrm{Hom}_{\mathbb Z}(L, \mathbb{Z})$ the dual lattice,
one gets a natural embedding $L \rightarrow L'$ by $l \mapsto (\cdot,l)$.
Furthermore, we can regard $L'$ as a subgroup of $L \otimes \mathbb{Q}$, therefore $(\cdot,\cdot)$
extends to $L'$ as well.
We introduce the anti-dual basis elements  $E_v^{\ast}$ in $L'$ defined by
 $(E_{v'}, E_v^{\ast})$ being $-1$ if $v = v'$ and $0$ otherwise.
Notice that $L' \cong H^2(P) \cong H_2(P, M)$.  The short exact sequence
$0\to H_2(P)\to H_2(P,M)\to H_1(M)\to 0$ identify $L'/L $ with $H_1(M)$, which will
be denoted by $H$.
We denote the \emph{class} of $l' \in L'$ by $[l'] \in H$, and we call
$l' \in L'$  a \emph{representative} of $[l']$.

% Sometimes we consider the following additional information on plumbing graphs.
%An \emph{arrowhead supported on vertex} $v \in \mathcal{V}$ is an extra edge with one endpoint being the vertex $v$ and the other endpoint being a new extra vertex denoted by an arrow in visual representation of the plumbing graph.
%We say that this arrowhead \emph{represents a cut} $S$ \emph{(supported on $E_v$)}. In the plumbing construction, we think of $S$ as a disc fiber transverse to the `core' surface $E_v$.
% Therefore, this `cut' can be regarded as an element in $H_2(P, M) \cong L'$ (likewise its integer multiples $nS$), so the intersection form also applies to it. ($S$ has intersection $1$ with $E_v$ and $0$ with other cores.)
% Also, it represents a knot in the plumbed 3-manifold $M$. If we obtain $G$ as a resolution graph of a singularity, the strict transform of a function is a linear combination of such `cuts'.

Assume that the  intersection form in the basis $\{E_v\}_v$ has matrix $I$; then
 we define  $\det(G):=\det(-I)$. It also equals the order of $H$.
 (Since $I$ is negative definite,
 $\det(G)>0$.)

For any $h \in H$, we denote by $r_h= \sum_{v \in \mathcal{V}} c_v E_v$ `the smallest
effective representative' of its class in $L'$, determined by the property
$0 \leq c_v < 1$ for all $v$.

Finally, we define the \emph{canonical characteristic element} in $L'$. It is the unique element $k_G \in L'$
such that $(k_G, E_v) = -(E_v,E_v)-2$ for every $v \in \mathcal{V}$. (In fact, $P$ carries the structure of a smooth complex surface --- in the case
of singularities, $P$ is a resolution, cf. Section 2 ---,  and
$k_G$ is the first Chern class of its complex cotangent bundle.)

The Seiberg--Witten invariants of $M$ associate a rational number
 to each spin$^c$ structure on $M$. There is a `canonical' spin$^c$ structure $\sigma_{\textrm{can}} \in {\rm Spin}^c(M)$, the restriction of that spin$^c$ structure of
 $P$, which has first Chern class $k_G\in H^2(P)$.
  As we assumed $M$ to be a $\mathbb{Q}HS^3$, ${\rm Spin}^c(M)$ is finite. It is an
  $H$ torsor: for $h \in H$, we denote this action by $\sigma \mapsto h \ast \sigma$.

We denote by $\frsw_{\sigma}(M) \in \mathbb{Z}$ the \emph{Seiberg--Witten invariant} of $M$ corresponding to the spin$^c$ structure $\sigma$. This is the classical monopole counting Seiberg--Witten invariant of $M$ corrected by the Kreck--Stolz invariant to make it dependent only on the manifold $M$.

Now we are ready to define the following invariant for each homology element $h \in H$:
\begin{equation}\label{eq:Intr1}
 \mathfrak{i}_h(M) := \frac{(k_G+2r_h, k_G+2r_h) + \#\mathcal{V}}{8}. \end{equation}
In fact, it does not depend on the particular plumbing representation of the manifold $M$ (or $P$); it is an invariant of the manifold $M$. Next we  define the following normalization
 of the Seiberg--Witten invariant: for any $h \in H $, we set
\begin{equation}\label{eq:Inrt2}
 \mathfrak{s}_h(M) = \frsw_{h \ast \sigma_{\textrm{can}}}(M) - \mathfrak{i}_h(M). \end{equation}
Sometimes we will also use the notations
$\mathfrak{s}_h(G) = \mathfrak{s}_h(M)$, or
$\frsw_h(G) = \frsw_{h \ast \sigma_\textrm{can}}(M)$.

In fact, $\mathfrak{s}_h(M) \in \mathbb{Z}$. This can be seen easily through the identity \eqref{eq:LCoh}. We also refer to
\S \ref{ss:latticedef} for the fact that $\mathfrak{s}_h(M)$ (and thus $\mathfrak{i}_h(M)$ as well) is indeed independent of the plumbing representation of the manifold.

Let $\Sigma$ be the universal abelian cover (UAC) of the manifold $M$: it is associated with the
abelianisation $\pi_1(M)\to H_1(M)$.  In the next definition, $0$
is the unit element in $ H_1(\Sigma)$.

\begin{defn}\label{def:s}
 We say that for a manifold $M$ \emph{the `covering additivity property' of the invariant} $\mathfrak{s}$ \emph{holds with respect to the universal abelian cover} (shortly,
 `CAP of $\mathfrak{s}$ holds') if
 \[ \mathfrak{s}_{0}(\Sigma) = \sum_{h \in H_1(M)} \mathfrak{s}_{h}(M). \]
\end{defn}
Our main result is the following.

\begin{theorem}\label{thm:main0}
 Let $M= S^3_{-p/q}(K)$ be a manifold obtained by a negative  rational Dehn
 surgery of $S^3$ along a connected sum of algebraic knots $K = K_1 \# \dots \# K_{\nu}$
 ($p,q>0$, ${\rm gcd}(p,q)=1$).
 Assume that $\Sigma$, the UAC  of $M$,  is a $\mathbb{Q}HS^3$. Then CAP of $\mathfrak{s}$ holds.
\end{theorem}

Though the statement is topological, in the proof we use several analytic steps
based on the theory of  singularities. These steps not only emphasize the
role of the algebraic knots and of the negative definite plumbing construction, but they also
provide the possibility to use certain deep results valid for singularities.

We emphasize that the above covering additivity property is not true for general
negative definite plumbed 3--manifolds (hence for general 3--manifolds either), cf.
Example \ref{rem:counter}.
In particular, we cannot expect a proof of the main theorem by a general topological machinery.

\section{Preliminaries}\labelpar{s:prelimi}

\subsection{Connection with singularity theory}\labelpar{ss:sing}

We present the connection of Theorem \ref{thm:main0}  with singularity theory,
namely, with the (equivariant) geometric genera of normal surface singularities and the Seiberg--Witten Invariant Conjecture of N\'emethi and Nicolaescu \cite{NemnicI}. For details
 we refer to \cite{Npq,Nsplice,NemSW,NemnicI,NemnicIII}.

Let $(X,0)$ be a complex normal surface singularity (germ) with link $M$.
Let $\pi:\widetilde{X}\to X$ be a good resolution
with negative definite dual resolution graph $G$, which can be regarded also as a plumbing graph for the 4-manifold $\widetilde{X}$ and its boundary $M$. (Hence, the $E_v$'s in this context are the irreducible exceptional curves.)
The \emph{geometric genus} of the singularity is defined as  $p_g(X) = \textrm{dim}_{\mathbb{C}} H^1(\widetilde{X}, \mathcal{O}_{\widetilde{X}})$, where $\mathcal{O}_{\widetilde{X}}$  is the
structure sheaf of $\widetilde{X}$.
It does not depend on the particular choice of the resolution.
In \cite{NemnicI} the following conjecture was formulated for certain singularities, as a topological characterization of $p_g(X)$:
\begin{equation}\label{eq:pg}
 p_g(X) = \mathfrak{s}_0(M).
\end{equation}
We say that the \emph{Seiberg--Witten Invariant Conjecture} (SWIC) holds for $X$ if
\eqref{eq:pg} is true.

It is natural to ask whether there is any similar connection involving the other Seiberg--Witten invariants? The answer is given in \cite{Npq,Nsplice}.
Let $(Y,0)$ be the universal abelian cover of the singularity $(X,0)$ (that is,
  its link  $\Sigma$  is the regular UAC of $M$, $(Y,0)$ is normal, and $(Y,0)\to (X,0)$ is
  analytic).
The covering  action of  $H = H_1(M)$ on $Y$ extends to the resolution $\widetilde{Y}$ of $Y$,
hence $H$ acts on  $H^1(\widetilde{Y}, \mathcal{O}_{\widetilde{Y}})$ as well,
  providing a eigenspace decomposition $\oplus _{\xi\in\widehat{H}}\,
  H^1(\widetilde{Y}, \mathcal{O}_{\widetilde{Y}})_{\xi}$, indexed
   by the characters $\xi\in\widehat{H}:= {\rm Hom}(H,{\mathbb C}^*)$ of $H$.
   Set
\[ p_g(X)_h = \textrm{dim}_{\mathbb{C}} H^1(\widetilde{Y}, \mathcal{O}_{\widetilde{Y}})_{\xi_h}, \]
where $\xi_{h} \in \widehat{H}$ is the character given by $h'' \mapsto e^{2\pi i (l', l'')}$,  $[l']=h$, $[l'']=h''$.  The numbers $p_g(X)_h$
 are called the \emph{equivariant geometric genera} of $(X,0)$.
 Note that $p_g(X)_0 = p_g(X)$.

We say that the \emph{Equivariant  Seiberg--Witten Invariant Conjecture} (EqSWIC) holds for
$(X,0)$ if the next identity \eqref{eq:pgh} is satisfied for \emph{every} $h \in H$:
\begin{equation}\label{eq:pgh}
 p_g(X)_h = \mathfrak{s}_h(M).
\end{equation}

\noindent
 Observe that by the definition, $p_g(Y) = \sum_{h \in H} p_g(X)_h$. Hence, the
 next claim is obvious.

 \begin{clm}\label{cl:pgadd}
  If for a singularity $(X,0)$ with $\mathbb{Q}HS^3$ link the EqSWIC holds, and for its
  (analytic) universal abelian cover $(Y,0)$ with $\mathbb{Q}HS^3$ link the SWIC holds, then for the link $M$ of $X$ the (purely topological) covering  additivity property of $\mathfrak{s}$ also holds.
 \end{clm}

\begin{example}\label{ex:lens}
 By \cite{Npq,NemnicII} the assumptions of Claim \ref{cl:pgadd} are satisfied {\it e.g.}
  by cyclic quotient
 and weighted homogeneous singularities,  hence the  CAP of $\mathfrak{s}$ holds
 for lens--spaces and Seifert rational homology sphere 3--manifolds.
Theorem \ref{thm:main0} proves CAP  for surgery manifolds, and Example \ref{rem:counter}
shows that CAP does {\it not} hold for arbitrary plumbed 3--manifolds.
\end{example}

It is convenient to extend the definitions (\ref{eq:Intr1}) and (\ref{eq:Inrt2}) for any
representative $l'\in L'$:
$$ \mathfrak{i}_{l'}(G):= \frac{(k_G+2l', k_G+2l')_G + \#\mathcal{V}(G)}{8} \ \ \ \mbox{and} \ \ \
 \mathfrak{s}_{l'}(M) = \frsw_{[l']}(G) - \mathfrak{i}_{l'}(G).$$
By a computation, for two representatives
$[l'_1] = [l'_2] = h \in H$
one has:
\begin{equation}\label{eq:difsh}\mathfrak{s}_{l_2'}(G)-\mathfrak{s}_{l_2'}(G)=
\chi(l'_2) - \chi(l'_1), \ \ \mbox{where} \ \ \
\chi(l') = -(l',l'+k_G)/2.\end{equation}
 In particular,
 \begin{equation}\label{eq:eulatgen}
 \mathfrak{s}_{l'}(G)  = \ \mathfrak{s}_{[l']}(G) + \chi(l')-\chi(r_{[l']}).
 \end{equation}
The invariants $\{\mathfrak{s}_h(G)\}_h$ for many 3--manifolds (graphs) are computed.
The next statement basically follows from Example \ref{ex:lens} combined
with the fact that the UAC of a lens space is $S^3$.
 \begin{prop}\label{prop:ratgraph}\cite{Npq,Nlatnorm,Nsplice}
 If $G$ is a (not necessarily minimal) graph of $S^3$ or of a lens--space then
 $\mathfrak{s}_h(G) = 0$ for every $h \in H$.
 \end{prop}

\subsection{The structure of the plumbing graph $G$ of $S^3_{-p/q}(K)$}\labelpar{ss:graphs}

In this section we describe
 the  plumbing graph of $S^3_{-p/q}(K)$ %and its UAC,
 and we also fix some additional notations.

Let $K_j\subset S^3$ be the embedded knot of an irreducible plane curve singularity
$\{f_j(x,y) = 0\}\subset ({\mathbb C}^2,0)$, where $f_j$ is a local holomorphic germ
$ (\mathbb{C}^2,0)\rightarrow (\mathbb{C},0)$.
Let $G_j$ be the
{\it minimal embedded resolution graph} of $\{f_j(x,y) = 0\}\subset ({\mathbb C}^2,0)$,
which is a plumbing graph (of $S^3$) with several additional decorations: it has an
arrowhead supported on a vertex $u_j$, which represents $K_j$ (or, in a different language,
the strict transform $S(f_j)$ of
$\{f_j=0\}$ intersecting the exceptional $(-1)$--curve $E_{u_j}$). Furthermore, $G_j$ has a set of
multiplicity decorations, the vanishing orders $\{m_v\}_v$ of the pullback of $f_j$
along the irreducible exceptional divisors and $S(f_j)$. We collect them  in the total transform
${\rm div}(f_j)=S(f_j)+\sum_{v\in {\mathcal V}(G_j)}m_vE_v =S(f_j) + (f_j)$ of $f_j$,
characterized by
$({\rm div}(f_j),E_v)_{G_j}=0$ for any $v$, and
 $(f_j)$ %:=\sum_{v\in {\mathcal V}(G_j)}m_vE_v$
 is its part supported on $\cup_{v\in\mathcal{V}(G_j)} E_v$. (For more on the graphs of plane curve singularities see
 \cite{BrKn,EN}.)
%Therefore, ${\rm div}(f_j) = S(f_j) + (f_j)$. \marginparred{Added a sentence.}

Next, we
write the surgery coefficient in \emph{Hirzebruch--Jung continued fraction}
\begin{equation}\label{eq:hj}
  p/q = k_0 - \cfrac{1}{k_1
          - \cfrac{1}{k_2
          - \cfrac{1}{\dots - \cfrac{1}{k_s} } } }\  =: [ k_0, k_1, \dots, k_s ],
\end{equation}
where $k_i\in {\mathbb Z}$,
$k_0\geq 1$, $k_1, \dots, k_s \geq 2$.
%We write $p/q = \langle  k_0, k_1, \dots, k_s\rangle$ for short.
Define $K := K_1 \# K_2 \# \dots \# K_{\nu}$.

Then $M=S^3_{-p/q}(K)$
can be represented by a negative definite plumbing graph $G$, which is constructed as follows,
cf. \cite{Npq,NR}.
 $G$ consists of $\nu$ blocks, isomorphic to $G_1, \dots, G_{\nu}$ (without the multiplicity
 decorations and arrowheads), a chain $G_0$ of length $s$ consisting of vertices $\overline{u}_1, \dots, \overline{u}_s = u'$ with decorations $e_{\overline{u}_1} = -k_1, \dots, e_{\overline{u}_s} = -k_s$, respectively, and one central vertex $u$ which is connected to the vertex $u_j$ in each block $G_j$ and to the first vertex $\overline{u}_1$ of the chain $G_0$ with decoration $-k_1$.
 The vertex $u$ has decoration $e_u = -k_0 - \sum_{j=1}^{\nu}m_{u_j}$.

 Note that if $q=1$ then $G_0$ is empty. In this case, we have $s=0$ and $u = u'$.

We use the notation $E_v$, $v \in \mathcal{V}(G)$, for the basis
 of the lattice $L_{G}$ associated with $G$.
We simply write $(.,.)$ for the intersection form $ (.,.)_G$,
 and $E_v^{\ast}$ for the anti-dual elements  in  $G$; that is, $(E_{v'},E_v^{\ast}) = -\delta_{v,v'}$ with the Kronecker-delta notation.

 Similarly, we write $(.,.)_j = (.,.)_{G_j}$ for the intersection form of $G_j$ ($j = 0, \dots, \nu$). For any $v \in \mathcal{V}(G_j)$, we set $E_v^{\ast,j}\in L'(G_j)$ for the anti-dual of $E_v$ \emph{in the graph} $G_j$; that is, $(E_{v'},E_v^{\ast,j})_j = -\delta_{v,v'}$ with the Kronecker-delta notation, $v'\in {\mathcal V}(G_j)$. %(Notice that $E_{u_j}^{\ast,j} = {\rm div}(f_j)-S(f_j)$.)

We denote the canonical class of $G$ by $k_G$ and the canonical class of $G_j$ by $k_{G_j}$.

 By a general fact of surgeries, $H_1(M) = H = \mathbb{Z}_p$. In fact, $[E_{u'}^{\ast}]$ is a generator of this group (see the proof of Lemma \ref{lem:UACgraph}).
Therefore $H=\{[hE_{u'}^{\ast}]\}_h$, where $h\in \{0, 1, \dots, p-1\}$.

\subsection{The structure of plumbing graph $\Gamma$ of the UAC $\Sigma $ of $M=S^3_{-p/q}(K)$}\labelpar{ss:graphs2} \
%Denote by $\Sigma$ the (topological) UAC of $M$. The plumbing graph of $\Sigma$ is denoted by $\Gamma$. Next we give a description of $\Gamma$.

We construct a plumbing graph $\Gamma $ as follows.
$\Gamma$ consists of $\nu$ blocks $\Gamma_1, \dots, \Gamma_{\nu}$ with distinguished vertices $w_1, \dots, w_{\nu}$, a chain $\Gamma_0$ of length $q-1$ consisting of vertices $\overline{w}_1, \dots, \overline{w}_{q-1} = w'$ all with decoration $-2$,  and a `central' vertex $w$ which is connected to vertices $w_j$ (one from each block $\Gamma_j$) and to $\overline{w}_1$ at one end of the chain $\Gamma_0$. If $q=1$ then $\Gamma_0$ is empty and $w = w'$. $\Gamma_j$ is a plumbing graph of the link of the suspension hypersurface singularity $\{g_j=0\}$, where
$g_j(x,y,z_j) = f_j(x,y) + z_j^p$ (for its shape see \cite{Nemsig}). The vertex  $w_j$ of $\Gamma_j$ is that
vertex which supports the arrowhead, if we regard $\Gamma_j$ as the
embedded resolution graph of $\{z_j = 0\} \subset \{g_j = 0\}$ (that is, it supports the strict transform  of $\{z_j=0\}$).

The self--intersection of $w$ is determined as follows.

 Let $F_v$, $v \in \mathcal{V}(\Gamma)$, denote the basis elements of the lattice $L(\Gamma)$
 associated with  $\Gamma$.

We write ${\rm div}(z_j)=S(z_j)+ \sum_{v \in \mathcal{V}(\Gamma_j)} n_v F_v$
 for the total transform of  $\{z_j=0\}$
 under the embedded resolution of $\{z_j = 0\} \subset \{g_j = 0\}$ with resolution graph $\Gamma_j$.
(${\rm div}(z_j)$  topologically is characterized  by
  $({\rm div}(z_j),F_v)_{\Gamma_j}=0$ for any
  $v\in{\mathcal V}(\Gamma_j)$; the strict transform $S(z_j)$ can be represented as an arrowhead on $w_j$.)
Then, the central vertex $w$ has decoration $e_w = -1-\sum_{j=1}^{\nu} n_{w_j}$ in $\Gamma$.

\begin{lemma}\label{lem:UACgraph}
% $\Gamma$ is a (possible) dual resolution  graph of the UAC $(Y,0)$.
 $\Gamma$ is a (possible) plumbing  graph of the UAC $\Sigma$.
\end{lemma}
\begin{proof}
Consider the following divisor $D$ supported on $L_G$.
On each $G_j$ it is
$(f_j)$, we put multiplicity 1 on $u$, multiplicity $k_0$ on
$\overline{u}_1$, and in general, the numerator of
$[ k_0, \ldots, k_{i-1}]$ on $\overline{u}_i$, $1\leq i\leq s$,
for notations see \S \ref{ss:graphs}. Furthermore, put
an arrowhead on $\overline{u}_s$ with multiplicity $p$. If this arrowhead represents a cut
$S$ supported by $E_{\overline{u}_s}$, then $pS+D$ has the property that
$(pS+D,E_v)=0$ for all $v\in {\mathcal V}(G)$, hence it is a topological analogue of the
divisor of a function. The algorithm which provides the (topological)
cyclic ${\mathbb Z}_p$--covering
of the plumbed 4--manifold $P$ with branch locus $pS+D$ is identical with the algorithm from
 \cite{NCyc,Nemsig} (which provides branched cyclic covers associated with analytic functions).

  The point is that $S$  has multiplicity $p$, hence the
${\mathbb Z}_p$--covering will have no branching along it, hence, in fact, the reduced
branch locus is in $\cup_vE_v$. Note that $D/p=E^*_{u'}$, and its $E_u$--coefficient is $1/p$,
hence the class of $E^*_{u'}$ (or of $D/p$)  has order $p$ in $H$, hence it  generates  $H$.
This implies that this algorithm provides exactly the UAC of $M$.
Since the algorithm is `local', and the multiplicity of $u$ is 1, over the subgraphs $G_j$ it is
identical with that one which provides the graph of the suspension
singularity $f_j+z_j^p$ (see again   \cite{NCyc,Nemsig}).

 Next we verify its behaviour over the graph $G_0$.
This graph is the graph of  a Hirzebruch--Jung singularity of type $(q,r)$, where
$q/r=[k_1,\ldots, k_s]$. (For details regarding Hirzebruch--Jung singularities see \cite{BPV}.)
This is the  normalization of $xy^{q-r}=z^q$. Using this coordinate choice,
the strict transform of $y$ is exactly $qS$, the strict transform of $x$ is a disc $S'$ in
$E_u$ (a disc neighbourhood of  $E_u\cap E_{\overline {u}_1}$ in $E_u$) with multiplicity $q$;
and finally, the strict transform of $z$ is $S'+(q-r)S$. In particular, the cyclic covering
we consider over $G_0$  is exactly the cyclic ${\mathbb Z}_p$--covering of the normalization of
$xy^{q-r}=z^q$ along the divisor of  $zy^{k_0-1}$ (here for the $S$--multiplicity use the identity $q-r +(k_0-1)q=k_0q-r=p$). This is a new Hirzebruch--Jung singularity, the normalization of
$xy^{q-r}=z^q$ and $zy^{k_0-1}=w^p$. The $q$--power of the second equation combined with the
first one gives $xy^p=w^{pq}$, hence $t:=w^q/y$ is in the integral closure with $x=t^p$.
Hence, after eliminating $x$, the new equations are $ty=w^q, \ t^py^{q-r}=z^q$ and $zy^{k_0-1}=w^p$.
A computations shows that the integral closure of this ring is given merely by $ty=w^q$.
This is an
$A_{q-1}$ singularity, whose minimal resolution graph is $\Gamma_0$.

Finally, notice that the above algorithm
provides a system of multiplicities, which can be identified with a homologically trivial
divisor, hence, similarly  as in   \cite{NCyc,Nemsig}, we get the last `missing Euler number'
$e_w$ too.
\end{proof}

The intersection form of $\Gamma$ will be denoted by $\langle.,.\rangle = (.,.)_{\Gamma}$. Similarly, $\langle.,.\rangle_j = (.,.)_{\Gamma_j}$ will denote the intersection form of $\Gamma_j$. The canonical class of $\Gamma$ is $k_{\Gamma}$, the canonical class of $\Gamma_j$ is $k_{\Gamma_j}$. For any $v \in \mathcal{V}(\Gamma)$, $F_v^{\ast}$ will denote the anti-dual of the corresponding divisor $F_v$ in $\Gamma$.
%(\textit{i.e.} $\langle F_{v'}, F_v^{\ast} \rangle = -\delta_{v,v'}$).
Similarly, for a vertex $v \in \mathcal{V}(\Gamma_j)$, $F_v^{\ast,j}$ is the anti-dual of $F_v$ in $\Gamma_j$.
Set $J = H_1(\Sigma)$
 and let $J_j$ be the first homology group of those 3--manifolds determined by $\Gamma_j$.

\begin{lemma}\label{lem:JJ}
 \[J \cong J_1 \times \dots \times J_{\nu}.\]
\end{lemma}

\begin{proof}
Let $\rho_j\in \widehat{J}_j$ be a character of $\Gamma_j$, $j\geq 1$. In \cite[\S 6.3]{NemnicIII}
is proved that $\rho_j$ takes value 1 on $F^{*,j}_{w_j}$ (recall that the vertex $w_j$
of $\Gamma_j$ is connected with
the central vertex $w$). Hence, for $j\not=i$, $j,i\geq 1$, there is no edge $(v_j,v_i)$
of  $\Gamma$, such that $v_j$ is in the support of $\rho_j$ and $v_i$ is in the support of
$\rho_i$. This means that each $\rho_j\in\widehat{J}_j$ can be extended to a character
of $J$, by setting  $\rho_j(F^*_v)=1$ whenever $v\not\in{\mathcal V}(\Gamma_j)$; in this way providing
a monomorphism $\widehat{J}_j\hookrightarrow \widehat{J}$. But the same property also guarantees
that in fact one has a simultaneous embedding
$\prod_{j\geq 1} \widehat{J}_j\hookrightarrow \widehat{J}$.

Therefore, if we prove that  $\prod_{j\geq 1}\det(\Gamma_j)=\det(\Gamma)$, then the above embedding becomes an isomorphism, hence Lemma follows.
In determinant computations of decorated trees, the following formula is useful;  see e.g.
\cite[4.0.1(d)]{NB}.

Let $e$ be an edge of $\Gamma$ with end vertices $a$ and $b$.
Then $\det(\Gamma)=\det(\Gamma\setminus e)-\det(\Gamma\setminus \{a,b\})$.

This formula inductively (applied for the edges adjacent to $w$) provides the needed determinant
identity.
\end{proof}

 \begin{lemma}\label{cl:eq}

\begin{equation}\label{eq:intform1}\begin{split}   %\label{eq:intform2}
(a) \ \ \ \ \,
- p\cdot (E_u^{\ast}, E_u^{\ast}) =& \ q \ \ \textrm{\ and\ } \ \  -p\cdot (E_{u'}^{\ast}, E_u^{\ast}) = 1;\\
(b) \ \ \ \    q\cdot(E_{u_j}^{\ast,j}, E_v^{\ast,j})_j = & \ p\cdot (E_u^{\ast},E_v^{\ast}) \ \ \
\mbox{for any  $v \in \mathcal{V}(G_j)$, $j \geq 1$};\\
(c) \ \ \ \ \ \ \ \
- \langle F_{w}^{\ast}, F_{w}^{\ast} \rangle = & \ q \ \ \textrm{\ and\ } \ \ -\langle F_{w'}^{\ast}, F_{w}^{\ast} \rangle = 1;\\
(d) \ \ \ \ q\cdot \langle F_{w_j}^{\ast,j}, F_v^{\ast,j} \rangle_j = & \
 \langle F_w^{\ast}, F_v^{\ast} \rangle
 \ \ \ \mbox{for any $v \in \mathcal{V}(\Gamma_j)$, $j \geq 1$}.\end{split}\end{equation}
\end{lemma}
\begin{proof}  For a negative definite plumbing
graph $\frG$ (which is a tree) if we define the anti--duals  $E^*_v$ as above,
then the following holds:
for any two vertices $a,b$ the expression $-\det(\frG)\cdot(E^*_a,E^*_b)$ equals the product of the determinants of the connected components of that graph which is obtained from $\frG$ by deleting
the shortest path connecting $a$ and $b$ and the adjacent edges;
 see \cite[\S 10]{EN} in the integral homology case and \cite{NemnicI} in general.
 % or Lemma 20.2 in EisNeu?

This applied for $G$ and $a=b=u$ (and $a=u$, $b=u'$)
 gives $(a)$, since $\det(G)=p$, $\det(G_j)=1$ for $j\geq 1$ and
$\det(G_0)=q$. (b) follows similarly.
(c) and (d) follows from this property combined with Lemma \ref{lem:JJ}.
\end{proof}

\section{Additivity property of the invariant $\mathfrak{s}$}\labelpar{s:main}

\subsection{Proof of the main theorem}\labelpar{ss:mainproof}

Now we are ready to prove Theorem \ref{thm:main0}. To adjust it to its proof, we recall it in a more explicit form, in the language of plumbing graphs.

\begin{theorem}\label{thm:main}
Let $K$ be a connected sum of algebraic knots, $p, q$ coprime positive integers. Assume that
 both $S^3_{-p/q}(K)$ (having plumbing graph $G$) and its
  universal abelian cover $\Sigma$ (with plumbing graph $\Gamma$)
 are  rational homology spheres.
  Then the following additivity holds:
\[ \underbrace{\frsw_{0}(\Gamma) - \frac{\langle k_{\Gamma},k_{\Gamma}\rangle + \#\mathcal{V}(\Gamma)}{8}}_{\mathfrak{s}_0(\Gamma)} =
   \sum_{h = 0}^{p-1} \underbrace{\left[ \frsw_h(G) - \frac{(k_G + 2r_h, k_G + 2r_h) + \#\mathcal{V}(G)}{8} \right]}_{\mathfrak{s}_h(G)}. \]
\end{theorem}
On the left hand side $0$ is the unit element of $J = H_1(\Sigma)$ and on the right hand side we identified elements of $H \cong \mathbb{Z}_p$ with elements of $\{0, 1, \dots, p-1\}$, $0$ being the unit element and $1$ being the generator $[E_{u'}^{\ast}]$, \textit{i.e.}, $r_h = r_{[hE_{u'}^{\ast}]}$.

In fact, the condition whether $\Sigma$ is a $\mathbb{Q}HS^3$ or not is readable already from $p$ and the plane curve singularity invariants describing the knots $K_i$; cf. \cite[\S 6.2 (c)]{NemnicIII}.

\begin{proof}

 Notice that deleting from $G$ the `central vertex' $u$ (and all its adjacent  edges), one gets $G_0, G_1, \dots, G_{\nu}$ as connected components of the remaining graph. Also,
 deleting from $\Gamma$ the `central vertex' $w$ (and all its  adjacent edges), one gets $\Gamma_0, \Gamma_1, \dots, \Gamma_{\nu}$ as connected components of the remaining graph. ($\Gamma_0$ and
 $G_0$ are present only if $q > 1$).

We use the notations  $R_j$, resp. $\widetilde{R}_j$ for the  `restriction'
homomorphisms  $L_{G}' \rightarrow L_{G_j}'$, resp.
$L_{\Gamma}' \rightarrow L_{\Gamma_j}'$, dual to the natural inclusions
  $L_{G_j} \rightarrow L_{G}$, resp. $L_{\Gamma_j} \rightarrow L_{\Gamma}$. They are
characterised by  $R_j(E_v^{\ast}) = E_v^{\ast,j}$, if $v \in \mathcal{V}(G_j)$ and $0$
otherwise, resp.
$\widetilde{R}_j(F_v^{\ast}) = F_v^{\ast,j}$, if $v \in \mathcal{V}(\Gamma_j)$ and $0$ otherwise.
 E.g.,  $R_j(k_G) = k_{G_j}$ and $\widetilde{R}_j(k_{\Gamma}) = k_{\Gamma_j}$
 (cf. \cite[Def. 3.6.1 (2)]{NB}).

 We can apply the surgery (`cut-and-paste') formula of \cite[Theorem 1.0.1]{NB} (note the sign difference due to the different sign convention about $\frsw$) and get the following two formulae. The new symbols
 $\mathcal{H}^{\textrm{pol}}_{u,h}(1)$ and   $ \mathcal{F}^{\textrm{pol}}_{w,0}(1)$
 are values of polynomials in $t=1$ as in \cite[\S 3.5]{NB};
 their definitions will be recalled later in \eqref{eq:defh} and \eqref{eq:deff}.
\begin{equation}\label{eq:swsis2}
   \mathfrak{s}_h(G) = \mathcal{H}^{\textrm{pol}}_{u,h}(1) +
   \mathfrak{s}_{R_0(r_h)}(G_0) + \sum_{j=1}^{\nu}
   %\underbrace{......}_{0}
   \mathfrak{s}_{R_j(r_h)}(G_j), %  + \chi_j(R_j(r_h)) \Big],
\end{equation}
\begin{equation}\label{eq:swuac2}
   \mathfrak{s}_0(\Gamma) = \mathcal{F}^{\textrm{pol}}_{w,0}(1) +
   % \underbrace{......}_{0}
   \mathfrak{s}_0(\Gamma_0) + \sum_{j=1}^{\nu} \mathfrak{s}_0(\Gamma_j).
\end{equation}
In \eqref{eq:swsis2}, for $j \geq 1$, $[R_j(r_h)] = 0 \in L_{G_j}'/L_{G_j}$, as the latter one is the trivial group $H_1(S^3)$. Hence, by (\ref{eq:eulatgen}),
  $\mathfrak{s}_{R_j(r_h)}(G_j)=\mathfrak{s}_0(G_j)+
  \chi_j(R_j(r_h))$, where  $\chi_j(x) := -\frac{1}{2}(x,x+k_{G_j})_j$.

 Furthermore, by Proposition \ref{prop:ratgraph},
 $\mathfrak{s}_0(\Gamma_0) = 0$,
 and  $\mathfrak{s}_0(G_j)  = 0$ for $j \geq 1$ (as $G_j$ is a plumbing graph for $S^3$).
 Therefore,  the desired equality $\mathfrak{s}_0(\Gamma) = \sum_{h=0}^{p-1} \mathfrak{s}_h(G)$ reduces to the proof of the following three lemmas.
 \end{proof}

 \begin{lemma}\label{lem:pol}
  \[  \sum_{h=0}^{p-1} \mathcal{H}^{\textrm{pol}}_{u,h}(1) = \mathcal{F}^{\textrm{pol}}_{w,0}(1). \]
 \end{lemma}

 \begin{lemma}\label{lem:susp}
  \[ \hspace*{2cm} \sum_{h=0}^{p-1} \chi_j(R_j(r_h)) = \mathfrak{s}_0(\Gamma_j) \ \ \ \
  (\mbox{for $j \geq 1$}).\]
 \end{lemma}

 \begin{lemma}\label{lem:lens}
  \[\sum_{h=0}^{p-1} \mathfrak{s}_{R_0(r_h)}(G_0)  = 0. \]
 \end{lemma}

    In the next paragraphs we recall the definition of $\mathcal{H}^{\textrm{pol}}_{u,h}$ and $\mathcal{F}^{\textrm{pol}}_{w,0}$ (following  \cite[\S 3.5]{NB})
      adapted to the present case and notations,  and
      then we provide the proofs  of  the lemmas.

     First, given a rational function $\mathcal{R}(t)$ of $t$, one defines its \emph{polynomial part} $\mathcal{R}^{\textrm{pol}}(t)$ as the unique polynomial in $t$ such that $\mathcal{R}(t) - \mathcal{R}^{\textrm{pol}}(t)$ is either $0$ or it can be written as a quotient of two polynomials of $t$ such that the numerator has degree strictly less than the denominator.

     Now $\mathcal{F}^{\textrm{pol}}_{w,0}$ and $\mathcal{H}^{\textrm{pol}}_{u,h}$ are polynomial parts of rational functions defined as follows.

      \begin{equation}\label{eq:defh}
      \mathcal{H}_{u,h}(t) = \frac{1}{p} \sum_{\varrho \in \widehat{H}} \varrho^{-1}(h)
       \prod_{v \in \mathcal{V}(G)}(1-\varrho([E_v^{\ast}])t^{-p\cdot (E_u^{\ast},E_v^{\ast})})^{\delta_v-2},
      \end{equation}
  where $\delta_v$ denotes the degree (number of adjacent edges) of a vertex $v\in \mathcal{V}(G)$.

    \begin{equation}\label{eq:deff}
     \mathcal{F}_{w,0}(t) = \frac{1}{|J|}\sum_{\varrho \in \widehat{J}}
       \prod_{v \in \mathcal{V}(\Gamma)}(1-\varrho([F_v^{\ast}])t^{-|J|\langle F_w^{\ast}, F_v^{\ast} \rangle })^{\widetilde{\delta}_v-2},
    \end{equation}
where $\widetilde{\delta}_v$ denotes the degree of a vertex $v\in \mathcal{V}(\Gamma)$.

\begin{proof}[Proof of Lemma \ref{lem:pol}.] Set
\begin{equation}\label{eq:h}
 \mathcal{H}_u(t) := \sum_{h=0}^{p-1}\mathcal{H}_{u,h}(t)
 =  \prod_{v \in \mathcal{V}(G)}(1-t^{-p\cdot (E_u^{\ast},E_v^{\ast})})^{\delta_v-2}.
\end{equation}
As taking polynomial parts of rational functions is additive, Lemma \ref{lem:pol} follows if we prove
\begin{equation}\label{eq:polid}
      \mathcal{H}_{u}(t^{|J|}) = \mathcal{F}_{w,0}(t).
\end{equation}
Let $\Delta_j = \Delta_{S^3}(K_j)$ be the Alexander polynomial of the knot $K_j$ (defined as in \cite[\S 2.6, (8)]{NemnicIII}, or \cite{EN}). Then, since
 $E_{u_j}^{\ast,j} = (f_j) \in L_{G_j}$,
\[ \frac{\Delta_j(t)}{1-t} = \prod_{v \in \mathcal{V}(G_j)}(1 - t^{-(E_{u_j}^{\ast,j}, E_v^{\ast,j})_j})^{\delta_v-2}.\]
Comparing \eqref{eq:h} with the above formula for the Alexander polynomials and using the identities (a), (b) of Lemma \ref{cl:eq} we get that

    \begin{equation}\label{eq:halex}
     \mathcal{H}_u(t) = \frac{\prod_{j=1}^{\nu}\Delta_j(t^q)}{(1-t)(1-t^q)}.
    \end{equation}

Recall that $J_j = L_{\Gamma_j}'/L_{\Gamma}$ is the first homology group of the manifold determined by $\Gamma_j$ and that $F_{w_j}^{\ast,j} = (z_j) \in L_{\Gamma_j}$.
Let $\Delta_{j,\Gamma}$ be the Alexander polynomial of the knot $K_{j,\Gamma}$ in the manifold of $\Gamma_j$ determined by $z_j=0$ (see \cite[\S 2.6, (8)]{NemnicIII}). That is,
    \[ \frac{\Delta_{j,\Gamma}(t)}{1-t} = \frac{1}{|J_j|} \sum_{\varrho_j \in \widehat{J_j}}
       \prod_{v \in \mathcal{V}(\Gamma_j)} (1 - \varrho_j([F_v^{\ast,j}]) t^{-\langle F_{w_j}^{\ast,j},F_v^{\ast,j}\rangle_j})^{\widetilde{\delta}_v-2}. \]
Recall that $J = J_1 \times \dots \times J_{\nu}$. Consequently, any character $\varrho \in \widehat{J}$ can be written as a $\nu$-tuple of characters, $\varrho = (\varrho_1, \dots, \varrho_{\nu})$ with $\varrho_j \in \widehat{J}_j = \textrm{Hom}(J_j, \mathbb{C}^{*})$.
Furthermore, for any $v \in \mathcal{V}(\Gamma_j)$, $\varrho([F_v^{\ast}]) = \varrho_j([F_v^{\ast,j}])$ and $\varrho([F_v^{\ast}]) = 1$ if $v = w$ or $v \in \Gamma_0$ as in that case $F_v^{\ast}$ represents the trivial element in $L_{\Gamma}'/L_{\Gamma}$
(see also the proof of Lemma \ref{lem:JJ}).

Comparing \eqref{eq:deff} with the above formula for the Alexander polynomials and using the identities (c), (d) of Lemma \ref{cl:eq} we get that, setting $s = t^{|J|}$,
    \begin{equation}\label{eq:falex}
     \mathcal{F}_{w,0}(t) = \frac{\prod_{j=1}^{\nu}\Delta_{j,\Gamma}(s^q)}{(1-s)(1-s^q)}.
    \end{equation}

    By \cite[Prop. 6.6]{NemnicIII} $\Delta_j = \Delta_{j,\Gamma}$, so via \eqref{eq:halex} and \eqref{eq:falex} we obtain \eqref{eq:polid}.
\end{proof}

\begin{proof}[Proof of Lemma \ref{lem:susp}]
% In this subsection, for an element $l' \in L_{G}'$, resp. $l' \in L_{\Gamma}'$, we write $l'|_{G_j} \in L_{G_j}'$, resp. $l'|_{\Gamma_j} \in L_{\Gamma_j}'$ for coordinatewise restriction with respect to the basis elements $E_v$, resp $F_v$, that is, if $l' = \sum_{v \in \mathcal{V}(G)}c_vE_v$ then $l'|_{G_j} = \sum_{v \in \mathcal{V}(G_j)}c_vE_v$.
For any element $l' = \sum_{v \in \mathcal{V}(G)}c_vE_v \in L_{G}'$, let
 $\lfloor l' \rfloor := \sum_{v \in \mathcal{V}(G)} \lfloor c_v \rfloor E_v $, resp.
$\{l'\}:=l'-\lfloor l' \rfloor$,
%$\{ l' \} = \sum_{v \in \mathcal{V}(G)} \{ c_v \} E_v \in L_{G}'$
denote the coordinatewise integer, resp. fractional part of $l'$ in the basis $\{E_v\}_v$.
We use this notation for other graphs as well.

Using the description of $E^*_{u'}$ in the proof of Lemma \ref{lem:UACgraph} we have
$$h E^*_{u'}=\sum_{j\geq 1} h\cdot (f_j)/p + hE_u/p+D_0 \ \ \ \ (0\leq h<p),$$
where $D_0$ is supported on $G_0$. Since $r_h=\{hE^*_{u'}\}$, we obtain
$$r_h=hE^*_{u'}-\sum_{j\geq 1}\lfloor h\cdot (f_j)/p\rfloor -\lfloor D_0\rfloor.$$
Since  $R_j(E^*_{u'})=0$, $R_j(E_v)=E_v$ for $v\in{\mathcal V}(\Gamma_j)$,
and $R_j(E_v)=0$ for $v\not\in({\mathcal V}(\Gamma_j)\cup u)$, we get
  \begin{equation}\label{eq:3216}
   R_j(r_h) = - \lfloor h\cdot (f_j)/p\rfloor.
  \end{equation}
As $\Gamma_j$ is the plumbing graph of a suspension hypersurface singularity
 $g_j(x,y,z_j) = f_j(x,y) + z_j^p = 0$ and, as it is proved in \cite{NemnicIII}, \emph{for such
suspension singularities the SWIC holds} (see \S \ref{ss:sing}), we have
$\mathfrak{s}_0(\Gamma_j) = p_g(\{g_j=0\})$. Hence, the statement of the Lemma is equivalent with
\[ p_g(\{g_j=0\}) = \sum_{h=0}^{p-1} \chi_j\left(- \lfloor h\cdot (f_j)/p \rfloor\right). \]
This geometric genus formula has major importance  even independently of the present application.
We separate the statement in the following Claim.

 \begin{clm}\label{cl:susppg}
  Let $f(x,y) \in \mathbb{C}\{x,y\}$ be the equation of an irreducible plane curve singularity.
  Let $G_f$ be the dual resolution graph of a
   good embedded resolution  of $f$, from which we delete the arrowhead (strict transform) of $f$ and all the multiplicities. Let $(f)$ be the part of the divisor of $f$ supported on the
  exceptional curves.   Then for any positive integer $p$
 the geometric genus of the suspension singularity  $\{ g(x,y,z) = f(x,y) + z^p = 0 \}$ is
  \[ p_g(\{g= 0 \}) = \sum_{h=0}^{p-1} \chi (-\lfloor h\cdot(f)/p\rfloor). \]
%  where $\chi(x) = -\frac{1}{2}(x,x+k_{G_f})$. % with the intersection form $(., .)$ of $G_1$.
 \end{clm}

\begin{rem}
A combinatorial formula (involving Dedekind sums) for the \emph{signature} of (the Milnor fibre of) suspension singularities was presented in \cite{Nemsig}. Recall that Durfee and Laufer type formulae imply that the geometric genus and the signature determine each other modulo the link
(see \textit{e.g.} \cite[Theorem 6.5]{Neminv} and the references therein).
Nevertheless, the above formula is of different type.
\end{rem}

 \begin{proof}
 Let $\phi:Z\to ({\mathbb C}^2,0)$ be the embedded resolution of $f$. Consider the
 ${\mathbb Z}_p$ branched covering  $c:(\{g=0\},0)\to  ({\mathbb C}^2,0)$, the restriction of
 $(x,y,z)\mapsto (x,y)$. Let $c_\phi:W\to Z$ be the pullback of $c$ via $\phi$ and let
 $\widehat{c}_\phi: \widehat{W}\to Z$ be the composition of the normalization
 $n:\widehat{W}\to W$ with $c_\phi$. Then $\widehat{W}\to W\to \{g=0\}$
 is  a partial resolution of $\{g=0\}$: although it might have some Hirzebruch--Jung singularities,
 since these are rational, one has $p_g(\{g=0\})=h^1({\mathcal O}_{\widehat {W}})$.
 On the other hand, we claim that
 \begin{equation}\label{eq:kollar}
 (\widehat{c}_\phi)_*({\mathcal O}_{\widehat{W}})=\oplus_{h=0}^{p-1}\,
  {\mathcal O}_Z(\lfloor h\cdot (f) /p\rfloor ).
 \end{equation}
 This follows basically from \cite[\S 9.8]{kollar}. For the convenience of the reader we sketch
 the proof.

 We describe the sheaves $(c_\phi)_*(\calO_W)$ and $(\widehat{c}_\phi)_*(\calO_{\widehat{W}})$
 in the neighbourhood $U$ of a generic point of the exceptional set $E$ of $\phi$. Consider such a
 point with local coordinates $(u,v)$, $\{u=0\}= E\cap U$, $(f)$ in $U$ is given by
 $u^m=0$.
 Consider the covering, a local neighbourhood of type $\{(u,v,z)\,:\, z^p=u^m\}$ in $W$.
 Then $\calO_{W,0}$ as $\C\{u,v\}$--module is $\oplus _{h=0}^{p-1}z^h\cdot \C\{u,v\}$.
 For simplicity we assume ${\rm gcd}(m,p)=1$.
 The ${\mathbb Z}_p$--action is induced by the monodromy on the regular part, namely by the permutation of the $z$--pages,  induced over the loop $u(s)=\{e^{2\pi i s}\}_{0\leq s\leq 1}$.
 This is the multiplication by $\xi:=e^{2\pi i m/p}$. Hence, $z^h\C\{u,v\}$ is the
 $\xi^h$--eigensheaf of  $(c_\phi)_*(\calO_W)$.

 If we globalize $z\C\{u,v\}$, we get a line bundle on $Z$, say $\calL$.
 Then the local representative of $\calL^p$ is $z^p\C\{u,v\}=u^m\C\{u,v\}=\C\{u,v\}(-(f))$.
 Hence $\calL^{p}$ is trivialized by $ f\circ \phi$. Since ${\rm Pic}(Z)=0$,
 $\calL$ itself is a trivial line bundle on $Z$.

 Next, we consider the normalization $\widehat{W}$. Above $U$ it is $({\mathbb C}^2,0)$
 with local coordinates $(t,v)$, and the normalization is $z=t^m$, $u=t^p$. In particular,
 $(\widehat{c}_\phi)_*(\calO_{\widehat{W},0})=\oplus _{h=0}^{p-1}t^h\cdot \C\{u,v\}$, where
 $\calf^{(h)}:=t^h\cdot \C\{u,v\}$ is the $e^{2\pi ih/p}$--eigensheaf.
 Set the integer $m'$ with $0\leq m'<p$ and $mm'=1+kp$ for certain $k\in {\mathbb Z}$.
  Then one has the following eigensheaf inclusions: $t^h\C\{u,v\}\supset z^{\{\frac{hm'}{p}\}p}
  \cdot \C\{u,v\}=\calL^{\{\frac{hm'}{p}\}p}|_U$. Hence, for some effective cycle $D$ we must have
  $t^h\C\{u,v\}= \calL^{\{\frac{hm'}{p}\}p}(D)|_U$. This, by taking $m$-power reads as
  $z^h\C\{u,v\}= z^{\{\frac{hm'}{p}\}pm} \C\{u,v\}(mD)$. This means that if
  $\{hm' /p\}=m_h/p$ and $mm_h=k_hp+h$ for certain integers $m_h$ and $k_h$, $0\leq m_h<p$,
  then the local equation of $mD$ is $z^{\{\frac{hm'}{p}\}pm-h}=
  z^{k_hp}$. Hence $D$ locally is given by $t^{k_hp}=u^{k_h}$.
  Since  $k_h=\lfloor mm_h/p\rfloor$, 
  the global reading of this fact  is
  $D=\lfloor m_h\cdot (f)/p\rfloor$. Hence
 $$(\widehat{c}_\phi)_*(\calO_{\widehat{W}})=\oplus _{h=0}^{p-1}\calL^{\{\frac{hm'}{p}\}p}
 (\lfloor m_h\cdot (f)/p\rfloor).$$
 Since $\calL$ is a trivial bundle, and $h\mapsto m_h$ is a permutation of $\{0,\ldots, p-1\}$,
 (\ref{eq:kollar}) follows.

 Next, from (\ref{eq:kollar}) we obtain  $p_g(\{g=0\})=\sum_h h^1({\mathcal O}_Z(
 \lfloor h\cdot (f) /p\rfloor )$. Set $D':=\lfloor h\cdot (f) /p\rfloor $.
 Then from the cohomological  exact sequence of the exact sequence of sheaves
 $0\to {\mathcal O}_Z\to {\mathcal O}_Z(D')\to {\mathcal O}_{D'}(D')\to 0$, and from
 $h^1({\mathcal O}_Z)=p_g(({\mathbb C}^2,0))=0$, we get
 $h^1({\mathcal O}_Z(D'))=h^1({\mathcal O}_{D'}(D'))$. Since by Grauert--Riemenschneider vanishing
 $ h^0({\mathcal O}_{D'}(D'))=0$, we have
 $$h^1({\mathcal O}_Z(D'))=-\chi({\mathcal O}_{D'}(D'))=-(D',D')+(D',D'+K)/2=\chi(-D').$$
 This ends the proof of the Claim.
 \end{proof}
 Moreover, the proof of Lemma \ref{lem:susp} is also completed.
 \end{proof}

\begin{proof}[Proof of Lemma \ref{lem:lens}] We observe two facts. First,
 from  the proof of Lemma \ref{lem:UACgraph} we obtain that
 $R_0(r_h)$ only depends on the value $p/q$ (and not on the blocks $G_j$, $j \geq 1$).
 It has the same expression even if we replace all the graph $G_j$ by the empty graph.
Second,
from the equations (\ref{eq:swuac2}) and (\ref{eq:swsis2}) and the discussion after it we get that
 under the validity of Lemmas  \ref{lem:pol} and  \ref{lem:susp}
 (what we already proved for any situation)
 the main  Theorem \ref{thm:main} (property CAP) is equivalent with Lemma
  \ref{lem:lens}.  Put these two together, the validity of \ref{lem:lens} is equivalent with
 the validity of CAP in the case when $G_j=\emptyset$ for all $j\geq 1$. But CAP for $G_0\cup\{u\}$
 is true by Claim \ref{cl:pgadd} and Example \ref{ex:lens}.
\end{proof}

\section{The  invariant $\mathfrak{s}_h$ and lattice cohomology}\labelpar{s:lattice}

\subsection{Lattice cohomology}\labelpar{ss:latticedef}
 The normalized SW invariant
 $\mathfrak{s}_h(G)$ can also be expressed as the Euler characteristic of
 the \emph{lattice cohomology}. The advantage of this approach is that it provides an alternative, completely elementary way to define $\mathfrak{s}$, as the definition of the lattice cohomology is purely combinatorial from the plumbing graph $G$.

 We briefly recall the definition and some facts about the lattice cohomology associated with a ${\mathbb Q}HS^3$ 3--manifold  with negative definite plumbing graph $G$.
 For more  see \cite{Nlatnorm,NR}.

% In short, the construction is the following.
Usually one starts with a lattice $\Z^s$ with fixed base elements $\{E_i\}_i$. This automatically
provides a cubical decomposition of $\R^s=\Z^s\otimes \R$: the 0--cubes are the lattice points
$l\in \Z^s$, the 1--cubes are the `segments' with endpoints $l$ and $l+E_i$, and more generally,
a $q$--cube $\square=(l,I)$ is determined by a lattice point $l\in \Z^s$ and a subset $I\subset \{1,\ldots,s\}$
 with $\# I=q$, and it has vertices at the lattice points $l+\sum_{j\in J}E_j$ for different $J\subset I$.

 One also takes a weight function $w:\Z^s\to \Z$ bounded below,
  and for each cube $\square=(l,I)$ one defines
 $w(\square):=\max\{w(v), \ \mbox{$v$ vertex of $\square$}\}$.  Then, for each integer $n\geq \min(w)$
 one considers the simplicial complex $S_n$ of $\R^s$, the union of all the cubes $\square$ (of any dimension)
 with $w(\square)\leq n$. Then the {\it lattice cohomology associated with $w$} is
 $\{\mathbb{H}^q(\Z^s,w)\}_{q\geq 0}$, defined by $\mathbb{H}^q(\Z^s,w):=\oplus _{n\geq \min(w)}
 H^q(S_n,\Z)$. Each $\mathbb{H}^q$ is graded (by $n$) and it is a $\Z[U]$--module, where
 the $U$--action consists of  the restriction maps induced by the inclusions $S_n\hookrightarrow S_{n+1}$.
 Similarly, one defines the {\it reduced cohomology associated with $w$} by
  $\mathbb{H}_{{\rm red}}^q(\Z^s,w):=\oplus _{n\geq \min(w)}
 \widetilde{H}^q(S_n,\Z)$. In all our cases $\mathbb{H}_{{\rm red}}^q(\Z^s,w)$ has finite $\Z$--rank.
The \emph{normalized Euler characteristic} of  $\mathbb{H}^*(\Z^s,w)$ is ${\rm eu}\,\mathbb{H}^*:=
-\min(w)+\sum_{q\geq 0}\, (-1)^q\,{\rm rank}_\Z\, \mathbb{H}^q_{{\rm red}}$. Formally, we also set
${\rm eu}\,\mathbb{H}^0:=-\min(w)+{\rm rank}_\Z\, \mathbb{H}^0_{{\rm red}}$.

Given a negative definite plumbing graph $G$ of a $\mathbb{Q}HS^3$ 3-manifold $M$ and a representative $l' \in L'$ of an element  $[l'] = h \in H $,
one works with the lattice $L = L_G = \mathbb{Z}\langle E_v \rangle_{v \in \mathcal{V}(G)}$
and weight function $L \ni l \mapsto -\frac{1}{2}(l, l + k_G + 2l')$.
 %=: \chi_{G}^{(l')}(l)$.
 The cohomology theory  corresponding to this weight function is
  denoted by $\mathbb{H}^{\ast}(G; k_G+2l')$.
  If for  $h \in H$ we  choose the minimal representative $r_h\in
L'$, then the cohomology theory $\mathbb{H}^{\ast}(G; k_G+2r_h)$ is in fact an invariant of the pair $(M,h)$ (\textit{i.e.} it does not depend on the plumbing representation) and thus can be denoted by $\mathbb{H}_h^{\ast}(M)$. It is proven in \cite{NemSW} that for any $l'\in L'$ and $h=[l']\in H$
 \begin{equation}\label{eq:LCoh}
  \mathfrak{s}_{l'}(G) = \textrm{eu\ } \mathbb{H}^{\ast}(G;k_G+2l') \ \ \mbox{and} \ \
  \mathfrak{s}_h(M) = \textrm{eu\ } \mathbb{H}_h^{\ast}(M). \end{equation}

\subsection{Lattice cohomology  of integral surgeries.}\labelpar{ss:intsurg}

The lattice cohomology of integral surgeries ($q=1$)
was treated in \cite{NSurg,Npq,NR,BodNem}.
%\marginpar{\tiny Would be nice to have such explicit formulae in the case of $q > 1$ as well...}
We use the notation $p/q = d \in \mathbb{Z}$. Clearly  $u = u'$.

 Let $\Delta_j(t)$ be the Alexander polynomial of the algebraic knot $K_j$ normalised by $\Delta_j(1)=1$.
Let $\delta_j$ be the Seifert genus of $K_j\subset S^3$ (or, the delta--invariant of the
corresponding plane curve singularity),  and write  $\delta:=\sum_{j= 1}^\nu \delta_j$.
 Set also $\Delta(t) = \prod_{j=1}^\nu\Delta_j(t)$
 and write it in a form $\Delta(t) = 1 + \delta(t-1) + (t-1)^2Q(t)$ with $Q(t) = \sum_{i=0}^{2\delta-2}q_it^i$.
 Note that $Q(1)=\Delta''(1)/2$.

\begin{example}\label{ex:SIS} The case of integral surgeries is especially important %\marginparred{especially?}
in singularity theory, since the links of \emph{superisolated singularities} (see \cite{Lue, LMN}) are of this type. They appear
as follows.
Let $f \in \mathbb{C}[x,y,z]$ be an irreducible homogeneous polynomial of degree $d$ such that
its zero set in $\mathbb{CP}^2$
 is a rational cuspidal curve; \textit{i.e.} $C = \{ f=0 \}$ is homeomorphic to $S^2$ and all the singularities of $C$ are locally irreducible.
 Let their number be $\nu$.
 Assume that there is no singular point on the projective line given by $z=0$. Then the equation $f(x,y,z) + z^{d+1} = 0$ in $(\mathbb{C}^3,0)$
determines an isolated complex surface  singularity with link homeomorphic to $S^3_{-d}(K)$,
where $K$ is the connected sum of algebraic knots given by the local topological types of the singularities on $C$. In this case, by genus formula, $(d-1)(d-2)=2\delta$,
a relation which connects $K$ with $d$.

\vspace{2mm}

However, we can take the surgery (and plumbed) manifold $M = S^3_{-d}(K)$ for any $K$ and with arbitrary $d>0$, even
without the `analytic compatibility' $(d-1)(d-2)=2\delta$.
\end{example}

\vspace{2mm}

Next we recall some results on lattice cohomology,
which will be combined with  the above proved CAP.
(They  will be very useful in fast computations of examples in the next section.)
In the next general discussion
the identity $(d-1)(d-2)=2\delta$ will not be assumed.

Write $s_h := hE_u^{\ast}$, then $r_h = \{ hE_u^{\ast} \}=\{s_h\}$, and set also
$c_h := \chi(r_h) - \chi(s_h)$.

From \cite[Theorem 7.1.1]{NR} (cf. also \cite[Theorem 3.1.3]{BodNem}) we know that
\[ \mathfrak{s}_{s_h}(G) = \textrm{eu\ } \mathbb{H}^{\ast}\left(S^3_{-d}(K);k_G+2s_h\right) = \sum_{\substack{n \equiv h (\textrm{mod\ } d)\\ 0\leq n \leq 2\delta-2}} q_n. \]
By the surgery (`cut-and-paste') formula \cite[Theorem 1.0.1]{NB} one has
  \[ \mathfrak{s}_{s_h}(G) = \mathcal{H}^{\textrm{pol}}_{u,h}(1) + \sum_{j=1}^{\nu}
 \mathfrak{s}_{R_j(s_h)}(G_j).\] %(G_j; R_j(k_{G}+2s_h)) = \mathcal{H}^{\textrm{pol}}_{u,h}(1), \]
 Since $R_j(s_h) = 0$ and $\mathfrak{s}_{0}(G_j) = 0$ (cf.  Prop.
  \ref{prop:ratgraph}) we get $\mathfrak{s}_{s_h}(G) = \mathcal{H}^{\textrm{pol}}_{u,h}(1)$, hence
  \[ \mathcal{H}^{\textrm{pol}}_u(1) = \sum_{h=0}^{p-1}
  \mathfrak{s}_{s_h}(G)=\sum_{n=0}^{2\delta-2} q_n = Q(1). \]
  This is related with the invariants $\mathfrak{s}_h(G)$ as follows.
From  (\ref{eq:LCoh}) and (\ref{eq:difsh})
  \[ \sum_{h=0}^{d-1} \mathfrak{s}_h(G) = \sum_{h=0}^{d-1}
   \mathfrak{s}_{s_h}(G)+ \sum_{h=0}^{d-1} (\chi(r_h) - \chi(s_h)) =
   \sum_{n=0}^{2\delta-2} q_n + \sum_{h=0}^{d-1} c_h=Q(1)+
  \sum_{h=0}^{d-1} c_h.\]
The identity  $p_g(\{g_j=0\})=\sum_h\chi_j(-\lfloor h\cdot (f_j)/d \rfloor)$ from
 Claim \ref{cl:susppg} has the  addendum
  \[  \sum_{j=1}^{\nu}  \chi_j
  (- \lfloor h\cdot (f_j)/d \rfloor) = \chi(r_h)-\chi(s_h)= c_h. \]
Indeed,   $s_h=hE^*_u = hE_u/d + \sum_j h(f_j)/d = hE_u/d + \sum_j h E_{u_j}^{\ast,j}$,  $r_h=hE_u/d+\sum_j\{h(f_j)/d\}$,
  $s_h-r_h= \sum_j\lfloor h(f_j)/d\rfloor$,  hence  $(s_h,s_h-r_h)=0$. Therefore
$\sum_j\chi_j
  (- \lfloor h\cdot (f_j)/d \rfloor)=\chi(r_h-s_h)=\chi(r_h)-\chi(s_h)$.

  In this way, $\sum_{h=0}^{d-1} \mathfrak{s}_h(G)$ can be computed easily.

\section{Examples  and applications}\labelpar{s:exandapp}

\begin{example}\label{ex:mainex}
 Consider the plumbing graph of a superisolated singularity corresponding to a curve
 of degree $d = 8$ with three singular points whose knots
 %curve with singularities given by multiplicity sequences $[6], [2_4], [2_2]$.
 % \marginparred{Multiplicity sequence not introduced here...}
% That is, $\nu = 3$, $d = 8$ and
$K_1, K_2, K_3$ are the torus knots of type $(6,7), (2,9), (2,5)$, respectively.

%\marginpar{{\bf PICTURE!!}}
 \begin{center}
 \includegraphics[width=12cm]{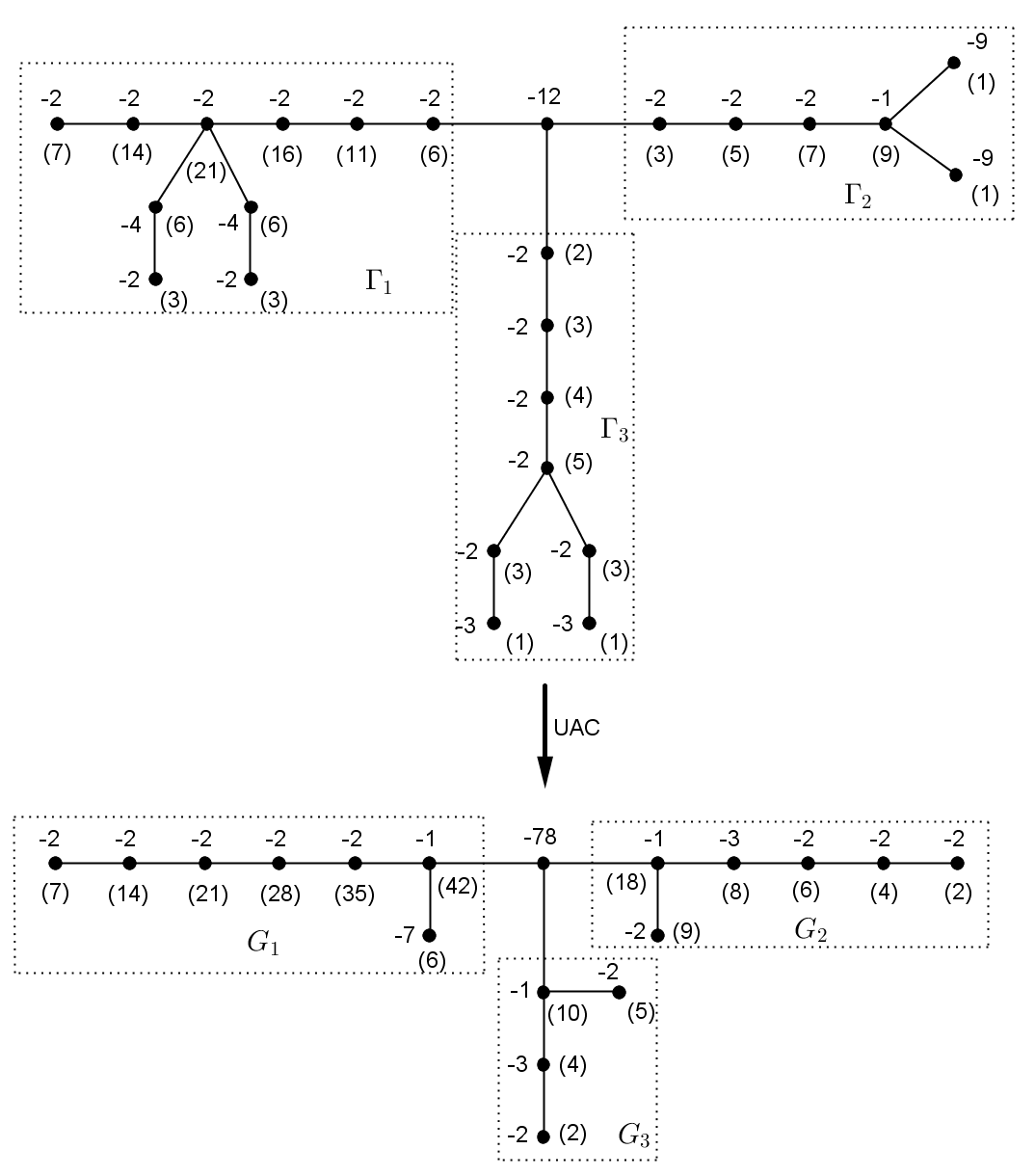}
 \end{center}

 One computes that
 \[ \mathcal{H}_{u}(t) =
    (1-t)\cdot \frac{1-t^{42}}{(1-t^6)(1-t^7)}\cdot \frac{1-t^{10}}{(1-t^2)(1-t^5)}\cdot
    \frac{1-t^{18}}{(1-t^2)(1-t^9)} =
    \frac{\Delta_1(t)\Delta_2(t)\Delta_3(t)}{(1-t)^2} \]
 and $\mathcal{H}^{\textrm{pol}}_u(1) = 293$. Correspondingly,
 $\sum_{h=0}^{7} \mathfrak{s}_{s_h}(G) = Q(1) = 293$. One also computes that $\sum_{h=0}^{7}c_h = 34$. Therefore, $\sum_{h=0}^{7}\mathfrak{s}_{h}(G) = 293 + 34 = 327$.

 After computing the graph $\Gamma$ of the UAC, we have $J = \mathbb{Z}_7 \times \mathbb{Z}_9 \times \mathbb{Z}_5$ and (setting $s = t^{7\cdot 9\cdot 5}$ and after summation $\Sigma_*$ over
 $\zeta_1 \in \mathbb{Z}_7$, $\zeta_2 \in \mathbb{Z}_9$,  $\zeta_3 \in \mathbb{Z}_5$, where $\mathbb{Z}_l = \{e^{\frac{2\pi i m}{l}} \}_m$ are cyclic groups),
the rational function  $\mathcal{F}_{w,0}(t)$ equals
 \[
 \frac{1-s}{7\cdot 9\cdot 5}\cdot  \sum_*
 %_{\zeta_1 \in \mathbb{Z}_7} \sum_{\zeta_2 \in \mathbb{Z}_9} \sum_{\zeta_3 \in \mathbb{Z}_5}
 \frac{(1-s^{21})^2}{(1-s^7)(1-\zeta_1^2s^3)(1-\zeta_1^{-2}s^3)}\frac{(1-s^9)}{(1-\zeta_2^5s)(1-\zeta_2^{-5}s)}\frac{(1-s^5)}{(1-\zeta_3s)(1-\zeta_3^{-1}s)} = \]
 \[ = \frac{\Delta_{1,\Gamma}(s)\Delta_{2,\Gamma}(s)\Delta_{3,\Gamma}(s)}{(1-s)^2}. \]
 Then  $\mathcal{F}_{w,0}(t) = \mathcal{H}_{u}(s)$ holds indeed with $s = t^{7\cdot 9\cdot 5}$. Correspondingly,
 \[\mathfrak{s}_0(\Gamma)
  = \mathcal{F}^{\textrm{pol}}_{w,0}(1) + \sum_{j=1}^{3}p_g(\{g_j=0\}) = 293 + 34 = 327. \]
\end{example}

\begin{example}\label{rem:counter}
 We wish to emphasize that the covering
 additivity property of $\mathfrak{s}$  is not true in general, not even when restricting ourselves to integral surgeries along algebraic knots in integral homology spheres (instead of $S^3$). This is shown by the next example (motivated by \cite[Remark 6.8.(2)]{NemnicIII}; the arrowhead of that graph is replaced by the $-8$ vertex below).

 \begin{picture}(300,100)(20,-10)

                         \put(84,72){\circle*{5}}

\put(60,48){\circle*{5}}  \put(84,48){\circle*{5}}

\put(60,24){\circle*{5}} \put(84,24){\circle*{5}}

                         \put(84,0){\circle*{5}}

\put(108,36){\circle*{5}} \put(132,36){\circle*{5}}

                            \put(84,48){\line(0,1){24}}

\put(60,48){\line(1,0){24}} \put(84,48){\line(2,-1){24}}

\put(60,24){\line(1,0){24}} \put(84,24){\line(0,-1){24}}

                            \put(84,24){\line(2,1){24}}

\put(108,36){\line(1,0){24}}

                                  \put(74,72){\makebox(0,0){$-5$}}

\put(50,48){\makebox(0,0){$-2$}}  \put(94,53){\makebox(0,0){$-1$}}

\put(50,24){\makebox(0,0){$-2$}} \put(84,32){\makebox(0,0){$-1$}}

                                 \put(74,0){\makebox(0,0){$-5$}}

\put(108,44){\makebox(0,0){$-7$}} \put(132,44){\makebox(0,0){$-4$}}

\put(150,36){\vector(1,0){24}}

\put(160,45){\makebox(0,0){UAC}}

\put(200,36){\circle*{5}} \put(224,36){\circle*{5}} \put(248,36){\circle*{5}} \put(272,36){\circle*{5}}

      \put(224,12){\circle*{5}}     \put(248,12){\circle*{5}}

\put(200,36){\line(1,0){72}}

\put(224,12){\line(0,1){24}} \put(248,12){\line(0,1){24}}

\put(200,44){\makebox(0,0){$-2$}} \put(224,44){\makebox(0,0){$-1$}} \put(248,44){\makebox(0,0){$-4$}} \put(272,44){\makebox(0,0){$-8$}}

      \put(214,12){\makebox(0,0){$-5$}}     \put(238,12){\makebox(0,0){$-2$}}

\end{picture}

If we replace
the $(-8)$--vertex of $G$ by an arrowhead (representing a knot $K$)
we get an integral homology sphere $\mathfrak{S}^3$,
the corresponding knot has $m_{u_1}=6$, hence $M(G)=\mathfrak{S}^3_{-2}(K)$, and
$G$ has determinant $2$. One computes that $\mathfrak{s}_0(G)+\mathfrak{s}_1(G)= 15 + 14 = 29$,
while $\mathfrak{s}_0(\Gamma) = 21$. In fact,  when trying to copy the proof of
Theorem \ref{thm:main}, one finds that neither the polynomial identity \ref{lem:pol}
holds (this is why the example was present in \cite[Remark 6.8.(2)]{NemnicIII}),
nor is $\Gamma \backslash w$ of suspension type (satisfying the SWIC).
\end{example}

\begin{remark}
 On the other hand, there are facts suggesting that the CAP of $\mathfrak{s}$ can hold in more general settings. Indeed, as we indicated in Claim \ref{cl:pgadd}, if for a given $M$ one can find a
 surface singularity  with link $M$ such that the EqSWIC holds for the singularity $(X,0)$ and the SWIC holds for its UAC, then the additivity of $\mathfrak{s}$ holds automatically.
 (Eq)SWIC was verified for many analytic structures,
 whose links are not of surgery type.
 On the other hand, the family of superisolated singularities is the main
 source of counterexamples for SWIC (and this was one of the motivations to test CAP for them).

Independently of any analytic argument,
one can also find purely topological examples
 for which  CAP still works
 (and in which cases not only that we cannot verify the presence of EqSWIC/SWIC, but we cannot even identify any specific analytic  structure on the topological type, or on certain special subgraphs).
% there is another simple example with det 3
Here is one  (for which the assumptions of Theorem \ref{thm:main} do not hold either).

 \begin{picture}(300,100)(20,-10)

                                                   \put(84,72){\circle*{5}}

\put(36,48){\circle*{5}} \put(60,48){\circle*{5}}  \put(84,48){\circle*{5}}

\put(36,24){\circle*{5}} \put(60,24){\circle*{5}} \put(84,24){\circle*{5}}

                                                   \put(84,0){\circle*{5}}

\put(108,36){\circle*{5}} \put(132,36){\circle*{5}}

                                                     \put(84,48){\line(0,1){24}}

\put(36,48){\line(1,0){24}} \put(60,48){\line(1,0){24}} \put(84,48){\line(2,-1){24}}

\put(36,24){\line(1,0){24}} \put(60,24){\line(1,0){24}} \put(84,24){\line(0,-1){24}}

                                                     \put(84,24){\line(2,1){24}}

\put(108,36){\line(1,0){24}}

                                                                   \put(74,72){\makebox(0,0){$-2$}}

\put(36,56){\makebox(0,0){$-4$}} \put(60,56){\makebox(0,0){$-3$}}  \put(94,53){\makebox(0,0){$-3$}}

\put(36,32){\makebox(0,0){$-4$}} \put(60,32){\makebox(0,0){$-3$}} \put(84,32){\makebox(0,0){$-3$}}

                                                                  \put(74,0){\makebox(0,0){$-2$}}

\put(108,44){\makebox(0,0){$-1$}} \put(132,44){\makebox(0,0){$-16$}}

\put(150,36){\vector(1,0){24}}

\put(160,45){\makebox(0,0){UAC}}

\put(200,36){\circle*{5}} \put(224,36){\circle*{5}} \put(248,36){\circle*{5}} \put(272,36){\circle*{5}} \put(296,36){\circle*{5}}

                          \put(248,12){\circle*{5}}     \put(272,12){\circle*{5}}

\put(200,36){\line(1,0){96}}

\put(248,12){\line(0,1){24}} \put(272,12){\line(0,1){24}}

\put(200,44){\makebox(0,0){$-4$}} \put(224,44){\makebox(0,0){$-3$}} \put(248,44){\makebox(0,0){$-3$}} \put(272,44){\makebox(0,0){$-1$}} \put(296,44){\makebox(0,0){$-32$}}

                                                                \put(238,12){\makebox(0,0){$-2$}}     \put(262,12){\makebox(0,0){$-2$}}

\end{picture}

%If we replace the $(-32)$--vertex by an arrowhead we get a graph, whose minimal representation
%is single vertex with decoration $-3$ (hence it represents a lens-space).
% Consider the embedded resolution graph of plane curve singularity with multiplicity sequence $[422]$. Add a vertex with decoration $-4$ and connect it to the leaf having decoration $-3$; and replace the arrow with a vertex with decoration $-32$.
 One verifies that $\det(G)=2$, and
  $\mathfrak{s}_0(G)+ \mathfrak{s}_1(G)= 147 + 132 = 279 = \mathfrak{s}_0(\Gamma)$.

This raises the interesting question to find the precise limits of the CAP.
\end{remark}

\begin{remark}\label{rem:END1}
The lattice cohomology plays an intermediate role connecting the analytic invariants
of a normal surface singularity $X$ with the topology of its link $M = M(G)$. \textit{E.g.}, one proves using 
\cite[Prop. 6.2.2, Ex. 6.2.3, Thm. 7.1.3, 7.2.4]{Nlatnorm} %\marginparred{!!!!!!!!!!!!!!!!!}
that for any $h$ one has
$$p_g(X)_h\leq
\textrm{eu\ } \mathbb{H}^0(M(G);k_G+2r_h).$$
Furthermore, for surgery manifolds $M(G) = S^3_{-d}(K)$ one has the vanishing
$\mathbb{H}^q(M(G),k_G+2r_h)=0$ for $q\geq \nu$ (\cite{NR}). In particular, for superisolated singularities corresponding to unicuspidal rational plane curves ($\nu=1$) one has
\begin{equation}\label{eq:REM1}
p_g(X)_h\leq
\textrm{eu\ } \mathbb{H}^*(M(G);k_G+2r_h)=\mathfrak{s}_h(M).\end{equation}
Therefore, for $M=S^3_{-d}(K)$ with $\nu=1$, if the SWIC holds for the UAC $(Y,0)$, that is, if
$p_g(Y)=\mathfrak{s}_0(\Sigma)$, then this identity, the CAP and (\ref{eq:REM1}) implies
$p_g(X)_h=\mathfrak{s}_h(M)$ for any $h$, that is, the EqSWIC for $(X,0)$.

This is important for the following reason: for superisolated singularities
we do not know (even at conjectural level) any candidate (either topological or analytic!)
for their
equivariant geometric genera. It is not hard to verify that $p_g(X)=d(d-1)(d-2)/6$, but no formulas
exist for $p_g(X)_h$, and no (topological or analytic) prediction exists for $p_g(Y)$ either.
\end{remark}

\begin{example}\label{ex:eq}
 Set $\nu = 1$, $d = 4$, and let $K_1$ be the $(3,4)$ torus knot. This can be realized by the
 superisolated singularity $zx^3+y^4+z^5=0$. In this case
 $M = S^3_{-4}(K_1)$.

One verifies that $\sum_{h=0}^{3} \mathfrak{s}_h(M) = 9 = \mathfrak{s}_0(\Sigma)$
correspondingly to Theorem \ref{thm:main}.

On the other hand, the UAC $(Y,0)$ of the \emph{singularity} is the Brieskorn singularity $x^3+y^4+z^{16}=0$, whose geometric genus is $p_g(Y) = 9$ too.
Hence, by the above remark, $p_g(X)_h=\mathfrak{s}_h(M)$ for any $h$.
\end{example}

\begin{remark}\label{rem:END2} (Continuation of \ref{rem:END1}.)
It is interesting that we have two sets of invariants, an analytic package
$(\,\{p_g(X)_h\}_h, p_g(Y)\,)$ and a topological one $(\, \{\mathfrak{s}_h(M)\}_h,\mathfrak{s}_0(\Sigma)\,)$,
and both of them satisfy the additivity property.
Nevertheless, in some cases, they do not agree. For example, if $p_g(Y)<\mathfrak{s}_0(\Sigma)$, then by (\ref{eq:REM1}) necessarily at least one of the inequalities in (\ref{eq:REM1}) is strict.
In particular, for both topological and analytical package the additivity property is stable, it is
never damaged, but the equality of the two packages in certain cases fails.
\end{remark}
\begin{example}\label{ex:ineq}
 Set again $\nu = 1$ and $d = 4$, but this time let $K_1$ be the $(2,7)$ torus knot. As usual
  $M = S^3_{-4}(K_1)$. By a computation $\sum_{h=0}^{3} \mathfrak{s}_h(M) = 10 = \mathfrak{s}_0(\Sigma)$.

 A suitable superisolated singularity is given by  $(zy-x^2)^2-xy^3+z^5=0$.
 By  \cite{LMN} (the end of section 4.5.)
the  universal abelian cover $Y$ satisfies the strict inequality   $p_g(Y) < 10$.
Therefore, $p_g(X)_h < \mathfrak{s}_h(M)$ for at least one $h$ (in fact, not for $h=0$).
\end{example}

\end{document}